\documentclass[letterpaper, 10 pt, conference]{ieeeconf}
\IEEEoverridecommandlockouts
 \usepackage{booktabs} 
\usepackage{cite}
\usepackage{algorithmic}
\usepackage{textcomp}
\usepackage{xcolor}
\usepackage{bm}
\usepackage{amsfonts}
\usepackage{amsmath,cases,amssymb}
\usepackage[english]{babel}
\usepackage{cite}
\usepackage{color}
\usepackage{nomencl}
\usepackage{url}
\usepackage{bm}
\usepackage{multirow}
\usepackage{hyperref}
\usepackage{color}
\usepackage{multirow}
\usepackage{mathtools}
\usepackage{etoolbox}
\newtheorem{myth}{Theorem}

\newtheorem{mydef}{Definition}

\newtheorem{myremark}{Remark}

\begin{document}

\title{\LARGE \bf
Revealing Decision Conservativeness Through Inverse Distributionally Robust Optimization
}

\author{Qi Li, Zhirui Liang, Andrey Bernstein, Yury Dvorkin
\thanks{Qi Li and Zhirui Liang contributed equally to this work and their names are ordered alphabetically. Qi Li, Zhirui Liang, and Yury Dvorkin are with the Johns Hopkins University, Baltimore MD, U.S. (e-mail: \{qli112, zliang31, ydvorki1\}@jhu.edu).
Andrey Bernstein is with the National Renewable Energy Laboratory, Golden CO, U.S. (e-mail: andrey.bernstein@nrel.gov)}%
}

\maketitle

\begin{abstract}
This paper introduces Inverse Distributionally Robust Optimization (I-DRO) as a method to infer the conservativeness level of a decision-maker, represented by the size of a Wasserstein metric-based ambiguity set, from the optimal decisions made using Forward Distributionally Robust Optimization (F-DRO). By leveraging the Karush-Kuhn-Tucker (KKT) conditions of the convex F-DRO model, we formulate I-DRO as a bi-linear program, which can be solved using off-the-shelf optimization solvers. Additionally, this formulation exhibits several advantageous properties. We demonstrate that I-DRO not only guarantees the existence and uniqueness of an optimal solution but also establishes the necessary and sufficient conditions for this optimal solution to accurately match the actual conservativeness level in F-DRO. Furthermore, we identify three extreme scenarios that may impact I-DRO effectiveness. Our case study applies F-DRO for power system scheduling under uncertainty and employs I-DRO to recover the conservativeness level of system operators. Numerical experiments based on an IEEE 5-bus system and a realistic NYISO 11-zone system demonstrate I-DRO performance in both normal and extreme scenarios.
\end{abstract}

\section{Introduction}
Real-world optimization problems often require dealing with parameter uncertainty. Ambiguity sets can model this uncertainty by encapsulating a range of plausible probability distributions for the uncertain parameters. Distributionally Robust Optimization (DRO) extends beyond a single empirical distribution to find solutions effective across multiple distributions within the ambiguity sets, facilitating the inclusion of a conservativeness level in the decision-making process \cite{mohajerin2018data}. 
The conservativeness level can be affected by factors such as the quality of available information and the potential impact of negative outcomes. In DRO models, this conservativeness level is quantified by the size of the ambiguity set.

Inverse Optimization (IO) aims to infer unknown parameters within a Forward Optimization (FO) problem based on the observed optimal solutions of FO. A comprehensive review of the IO methods can be found in \cite{chan2023inverse}. IO is widely applied across various fields, including power systems and power markets. Studies such as \cite{ruiz2013revealing} and \cite{chen2017strategic} employed IO to reconstruct the piece-wise linear cost functions of generators from historical market-clearing data. Similarly, \cite{liang2023data} developed a data-driven IO model to estimate generator costs based on noisy data. These papers assume that all parameters, except for offer prices, are known. Conversely, \cite{birge2017inverse} applies IO to recover undisclosed market structures, like transmission line parameters. Additionally, IO has been used to estimate demand response parameters in power systems, as discussed in  \cite{saez2016data} and \cite{kovacs2021inverse}. However, these studies primarily focus on recovering specific unknown parameters within power systems, rather than abstract features of the decision-making model, such as the conservativeness level of the decision-maker.

Most research on the classical IO approach focuses on convex FO problems \cite{ahuja2001inverse,chan2020inverse,aswani2018inverse,mohajerin2018dataio}, with some exceptions exploring inverse discrete models \cite{schaefer2009inverse}. This paper builds upon the foundation of classical IO but targets a more specialized scenario by considering a DRO with a Wasserstein metric-based ambiguity set as the forward problem, where the value of the Wasserstein metric directly reflects the level of conservativeness. It effectively bridges the gap between Wasserstein metric-based DRO problems and the classical IO approach.

The motivation of this research is to replicate or predict actions of the decision-maker based on the Forward DRO (F-DRO) framework. Given the subjective nature of conservativeness, even individuals with identical prior knowledge may differ in their conservative decisions. 
To this end, we introduce an Inverse DRO (I-DRO) approach to recover the value of the Wasserstein metric from optimal decisions generated by F-DRO. 
This method allows external parties to more accurately forecast the decision-maker’s future actions, thereby enhancing their predictive capability and increasing their potential for profit.

The I-DRO framework has many potential applications within power systems. For example, consider a power market clearing problem solved by the market operator as the F-DRO problem. Market participants, aiming to maximize their profits, could employ a bi-level optimization framework to design optimal offer curves, as discussed in \cite{hobbs2000strategic}. Using I-DRO, market participants can estimate the decision conservativeness of the market operator. With this knowledge, they could replicate the market clearing process within the lower level of their own decision-making model, potentially leading to better-designed offer curves and increased profits. Although stochastic optimization, including DRO, has not been extensively implemented in current power system scheduling and market clearing processes, there are practical instances, such as \cite{abbaspourtorbati2015swiss}, and numerous theoretical studies, such as \cite{mieth2020risk}, that demonstrate its potential. 

To the best of our knowledge, this study is the first to propose an I-DRO framework. It differs from previous studies of the \textit{robust inverse optimization} (RIO)\cite{ghobadi2018robust} or \textit{distributionally robust inverse optimization} (DRIO) methods \cite{dong2021wasserstein}, which both aim to impute the unknown parameters in the FO model that are robust against multiple observed solutions. 
In these studies, the FO models could be either a linear or quadratic program, whereas the IO models entailed robust optimization \cite{ghobadi2018robust} or DRO \cite{dong2021wasserstein}.
A more closely related work is \cite{chassein2018variable}, designated as \textit{Inverse Robust Optimization} (I-RO), which determines the size of the uncertainty set in a forward robust optimization problem that renders a specific solution optimal.
Our contributions in this paper include:
\begin{itemize}
    \item Introducing an I-DRO model to recover conservativeness levels of the decision-maker from optimal decisions derived via F-DRO. 
    \item Providing rigorous proof of the existence and uniqueness of the optimal solution to I-DRO.
    \item Establishing the necessary and sufficient conditions to ensure the correctness of the conservativeness levels recovered by I-DRO and identifying three extreme scenarios that may lead to the failures of I-DRO. 
    \item Evaluating the performance of I-DRO within the power system domain, with the chance-constrained DC optimal power flow model as F-DRO. 
\end{itemize}

Following notations are used in this paper: $\Xi$ is the domain of a random variable $\xi$ (i.e., $\xi \in \Xi$), $\mathcal{M}(\Xi)$ contains all distributions defined on $\Xi$.

\section{Overview of Inverse Optimization}
\label{sec:IO_review}
Consider a convex problem where both the objective function and constraints depend on a decision variable $x$ and a parameter $\omega$ ($x$ and $\omega$ can be scalars or vectors). We define our FO problem as:
\begin{subequations} \label{general convex programming}
    \begin{align}
        \textbf{FO}(\omega):  \min_{x} \quad   &f(x,\omega) \\
        \mathrm{s.t.} \quad (\nu_i)  \quad &g_i(x,\omega) \leq 0,\  i= 1, \dots, m, \label{inequality constraint}
        % (\nu) & \quad A(p)x -b(p) = 0, \label{equality constraint}
    \end{align}
\end{subequations}
where $f$ and $\{g_i\}_{i=1}^{m}$ are differentiable and convex functions in $x$, and $\{\nu_i\}_{i=1}^{m}$ are dual variables associated with \eqref{inequality constraint}. We define $\mathcal{X}(\omega)$ as the set of feasible solutions and $\mathcal{X}^{\text{opt}}(\omega)$ as the set of optimal solutions to FO($\omega$), respectively, with the assumption that both sets are non-empty.

Consider a scenario where an individual is unaware of the value of $\omega$ but observes a decision $x^0$ resulting from FO($\omega$). This individual's goal is to recover the value of $\omega$ using IO based on the observed decision $x^0$. In a nominal setting, this goal is achieved by finding an $\omega$ such that $x^0 \in \mathcal{X}^{\text{opt}}(\omega)$, i.e., the Karush-Kuhn-Tucker (KKT) conditions of FO($\omega$) in \eqref{general convex programming} are satisfied. Thus, we formulate this IO problem as:
\begin{subequations} \label{convex KKT}
    \begin{align} 
        \textbf{IO-nominal}  &(x^0): \min_{\substack{x,\omega,\nu_i}} \ \left\Vert \omega-\bar{\omega} \right\Vert \label{FO_kkt:obj}\\
         \quad \mathrm{s.t.}   \quad &\nabla_x f(x,\omega) + \sum\nolimits_{i=1}^{m} \nu_i \nabla_x g(x,\omega) =0 \label{FO_kkt:stat}\\
        \quad &\nu_i g(x,\omega) = 0 ,\  i= 1, \dots, m \label{FO_kkt:comp}\\
        \quad & g_i(x,\omega) \leq 0,\  i= 1, \dots, m  \label{FO_kkt:prim}\\
        \quad & \nu_i \ge 0 ,\  i= 1, \dots, m \label{FO_kkt:dual} \\
        \quad & \omega \in \Omega, \label{FO_kkt:omega} \\
        \quad & x = x^0, \label{FO_kkt:x} 
    \end{align}
\end{subequations} 
where \eqref{FO_kkt:obj} seeks for $\omega$ that minimizes the distance to a predefined value $\bar\omega$ (an initial belief or estimate regarding $\omega$), \eqref{FO_kkt:stat} is the stationarity constraint, \eqref{FO_kkt:comp} ensures complementary slackness, \eqref{FO_kkt:prim} embodies primal feasibility, \eqref{FO_kkt:dual} denotes dual feasibility, $\Omega$ is the feasible region for parameter $\omega$. 

The IO problem in \eqref{convex KKT} assumes that the primal solution $x$ is fully observable. However, real-world situations may only permit partial observability of vector $x$, i.e., only a subset of elements from vector $x^0 \in \mathcal{X}^{\text{opt}}(\omega)$ can be observed. The observable elements are represented as $\{{x}^0_{i} \}_{\forall i \in \mathcal{N}}$, where $\mathcal{N}$ collects indices of the observable elements. The IO problem that accommodates partial observations, termed as partially constrained inverse problem, can be formulated as:
\begin{subequations} \label{partial IO}
    \begin{align}
         \textbf{IO-partial} (\{{x}^0_{i}\}_{\forall i \in \mathcal{N}}):  &\min_{\substack{x, \omega,\nu_i}} \quad  \left\Vert \omega-\bar{\omega} \right\Vert \label{FO_kkt_partial:obj} \\
        \mathrm{s.t.}  \quad & \eqref{FO_kkt:stat} - \eqref{FO_kkt:omega}  \nonumber \\
         \qquad & x_{i} = {x}^0_{i}, \ \forall i \in \mathcal{N}. 
    \end{align}
\end{subequations} 

\section{Forward DRO Formulation}
\label{sec:Forward_model}
This section introduces a standard chance-constrained linear program that incorporates a random variable. We then define an ambiguity set using the Wasserstein metric and subsequently reformulate the chance-constrained problem into a convex DRO problem. This convex DRO formulation is employed as the forward problem throughout the paper.
\subsection{Reformulation of Chance-Constrained Linear Program}
Consider the following linear program with $N^{oc}$ ordinary constraints and $N^{cc}$ chance constraints:
\begin{subequations} \label{general LP with CC}
    \begin{align}
        \min_{x} \quad & \quad  c^\top x \label{general objective}\\
        \mathrm{s.t.} \quad & \quad Ax \leq h \label{general normal constraint}\\
         & \quad \inf_{\mathbb{P} \in \mathcal{B}_{\epsilon}} \text{Pr}\left\{ Bx+D\xi \leq d \right\} \geq 1-\gamma,  \label{general CC constraint}
    \end{align}
\end{subequations}
where $x \in \mathbb{R}^{n}$ is a decision vector, and $\xi \in \mathbb{R}^{m}$ is a random variable following distribution $\mathbb{P}$.
The parameters $A \in \mathbb{R}^{N_{oc} \times n}$, $h \in \mathbb{R}^{N_{oc}}$, $B \in \mathbb{R}^{N_{cc} \times n}, D \in \mathbb{R}^{N_{cc} \times m}$, $d \in \mathbb{R}^{N_{cc}}$ are constant. 
The joint chance constraint in \eqref{general CC constraint} ensures that the probability of $Bx+D\xi$ being less than $d$ is at least $1-\gamma$, under the variability of $\xi$'s distribution $\mathbb{P}$ within the ambiguity set $\mathcal{B}_{\epsilon}$.

Chance constraint \eqref{general CC constraint} is hard to solve directly, so we reformulate it using conditional value-at-risk (CVaR) and Wasserstein metric-based ambiguity sets. 
CVaR is preferred over other risk measures as it is convex, translation-equivariant, and positively homogeneous \cite{rockafellar2000optimization}, while Wasserstein ambiguity sets outperform other formulations, like moment-based sets, in out-of-sample tests \cite{mohajerin2018data}.

We first give the definition of the Wasserstein metric, a measure of distance between two probability distributions:
\begin{mydef}    
(\textbf{Wasserstein Distance}) The Wasserstein metric $d_{W}(\mathbb{P}_1, \mathbb{P}_2): \mathcal{M}(\Xi) \times \mathcal{M}(\Xi) \rightarrow \mathbb{R}$ is defined by:
    \begin{equation} \nonumber
        d_{W}(\mathbb{P}_1,\mathbb{P}_2) = \inf_{\pi\in \Pi(\mathbb{P}_1,\mathbb{P}_2)}\left\{ \int_{\Xi^2}\Vert \xi_1-\xi_2 \Vert\pi(d\xi_1,d\xi_2) \right\},
    \end{equation}
where $\Pi(\mathbb{P}_1,\mathbb{P}_2)$ is the set of joint probability distributions of $\xi_1,\xi_2$ with marginal distributions of $\mathbb{P}_1$ and $\mathbb{P}_2$ \cite{mohajerin2018data}.
\end{mydef}

Given $N_s$ historical samples $\{ \hat{\xi_i}\}_{i=1}^{N_s}$, the empirical distribution $\hat{\mathbb{P}}_{N_s}$ can be derived as:
\begin{equation} \label{empirical distribution}
    \hat{\mathbb{P}}_{N_s} = \frac{1}{N_s} \sum\nolimits_{i=1}^{N_s} \delta_{\hat{\xi}_i},
\end{equation}
where $\delta_{\hat{\xi}_i}$ is a Dirac distribution centered at $\hat{\xi}_i$. The ambiguity set based on these historical samples is defined as a Wasserstein ball centered at $\hat{\mathbb{P}}_{N_s}$: 
\begin{equation} \label{ambiguity set}
    \mathcal{B}_{\epsilon}(\hat{\mathbb{P}}_{N_s}) = \left\{ \mathbb{P} \in \mathcal{M}(\Xi) \left\vert d_{W}\left(\mathbb{P}, \hat{\mathbb{P}}_{N_s}\right) \right.\leq \epsilon \right\},
\end{equation}
where $\epsilon$ is the radius of the Wasserstein ball.  
The value of $\epsilon$ determines the level of decision conservativeness, reflecting a decision-maker's confidence in the data quality \cite{mieth2023data}. It is generally considered to be confidential information belonging to the decision-maker.
A larger $\epsilon$ incorporates a broader range of possible distributions in the decision making model, potentially resulting in decisions that are more conservative and thus robust against less probable realizations of uncertainty. 

Following the method used in \cite{mieth2023data}, we can initially transform \eqref{general CC constraint} into the following equivalent form:
\begin{equation} \label{simplified cc 1}
    \inf_{\mathbb{P} \in \mathcal{B}_{\epsilon}}\text{Pr} \left\{ \max_{k=1,...,N_{cc}} \left[ a_{k}^{\prime\top} \xi + b_{k}^{\prime} \right] \leq 0\right\} \geq 1- \gamma,
\end{equation} 
where $a_{k}^{\prime}$ is the $k$-th row of matrix $D$, and $b_{k}^{\prime}$ is the $k$-th entry of vector $Bx-d$. The CVaR-based reformulation of \eqref{simplified cc 1} is given by:
\begin{equation} \label{simplified CC 2}
    \inf_{\tau} \mathbb{E}^{\mathbb{P}} \left[\tau + \frac{1}{\gamma} \max_{k=1,...,N_{cc}} \left[ a^{\prime\top}_k \xi + b^{\prime}_k -\tau \right]_{+}\right] \leq 0, 
\end{equation}
where $\tau \in \mathbb{R}$ is an auxiliary variable, and $\left[x\right]_{+}$ is an operator returning $\max\left\{x,0\right\}$. It is proved in \cite{rockafellar2000optimization} that enforcing \eqref{simplified CC 2} implies that \eqref{simplified cc 1} also holds. Then, we can transform \eqref{simplified CC 2} into the following distributionally robust form:
% following equivalent form using CVaR:
\begin{equation}
    \label{simplified CC 3}
    \inf_{\tau} \tau + \frac{1}{\gamma} \underbrace{ \sup_{\mathbb{P} \in \mathcal{B}_{\epsilon}(\hat{\mathbb{P}}_{N_s})}\mathbb{E}^{\mathbb{P}} \left[ \max_{k=0,1,...,N_{cc}} \overline{a}_{k}^{\top} \xi + \overline{b}_k \right]}_{\text{Inner supremum problem}} \leq 0. 
\end{equation}
This reformulation is derived by introducing $\overline{a}_{0} = \mathcal{O}_{m}$, $\overline{b}_{0} = 0$, where $\mathcal{O}_{m}$ denotes a all-zero vector of length $m$, as well as $\overline{a}_{k} = a_{k}^{\prime}$ and $\overline{b}_{k} = b_{k}^{\prime} - \tau$, $k=1,...,N_{cc}$. 

The inner supremum problem in \eqref{simplified CC 3} aims to identify the worst-case expectation within the ambiguity set $\mathcal{B}_{\epsilon}(\hat{\mathbb{P}}_{N_s})$. This problem is hard to solve directly due to its infinite dimension but can be reformulated to a convex form using the Wasserstein metric \cite{mohajerin2018data}. Consequently, we reformulate the inner supremum problem in \eqref{simplified CC 3} as:
\begin{subequations} \label{DRO convex reformulation}
    \begin{align}
        \inf_{\lambda, s_i, \tau} \quad & \quad \sum\nolimits_{j=1}^{m} \left( \lambda_{j} \epsilon + \frac{1}{N_s} \sum\nolimits_{i=1}^{N_s} s_{ij} \right) \\
        \mathrm{s.t.} \quad &  \quad s_{ij} \geq  a_{kj} \hat{\xi}_{ij} + b_k \\
         & \quad s_{ij} \geq a_{kj} \overline{\xi}_j + b_k - \lambda_{j} \vert \overline{\xi}_{j} - \hat{\xi}_{ij}\vert \\
         & \quad s_{ij} \geq a_{kj} \underline{\xi}_{j} + b_k - \lambda_{j} \vert \underline{\xi}_{j} - \hat{\xi}_{ij} \vert \\
         & \quad \lambda_{j} \geq 0\\
         &\quad i = 1,\dots,N_s, \ j = 1,\dots,m, \ k = 0,\dots,N_{cc}, \nonumber 
    \end{align}%
\end{subequations}%
\allowdisplaybreaks[0]%
where $a_{kj}$ is the $j$-th coordinate of $\overline{a}_{k}$, $b_k = \overline{b}_k/m$, $\underline{\xi}_j$ and $\overline{\xi}_j$ are the lower and upper bounds, respectively, of the $j$-th coordinate of $\Xi$ ($\xi \in \Xi$), and $\lambda_{j}$ is an auxiliary variable.

\subsection{Convex DRO Formulation}
By substituting the chance constraint in \eqref{general CC constraint} with the reformulations in \eqref{simplified CC 3} and \eqref{DRO convex reformulation}, we convert \eqref{general LP with CC} into the following finite-dimensional convex problem:
\begin{subequations} \label{forward DRO}
    \begin{align}
        \textbf{F-DRO}&(\epsilon): \min_{x,\tau,\lambda_{j},s_{ij}} \quad   c^{\top} x \\ 
        \mathrm{s.t.} \ & (\theta) \quad Ax \leq h \label{primal feasibility start} \\
         & (\mu) \ \ \tau +  \frac{1}{\gamma}\sum\limits_{j=1}^{m} \left( \lambda_{j} \epsilon + \frac{1}{N_s} \sum\limits_{i=1}^{N_s} s_{ij}\right) \leq 0 \label{CVaR constraints}\\
         & (\phi_{ikj}^{1}) \  a_{kj} \hat{\xi}_{ij} + b_k - s_{ij} \leq 0 \label{ab1}\\
         & (\phi_{ikj}^{2}) \ a_{kj} \overline{\xi}_{j} + b_k - \lambda_{j} \vert \overline{\xi}_{j} - \hat{\xi}_{ij}\vert -s_{ij} \leq 0 \label{ab2} \\
         & (\phi_{ikj}^{3}) \ a_{kj} \underline{\xi}_{j} + b_k - \lambda_{j} \vert \underline{\xi}_{j} - \hat{\xi}_{ij}\vert -s_{ij} \leq 0 \label{ab3}\\
         & (\eta_{j}) \quad \lambda_{j} \geq 0\label{primal feasibility end} \\
         & i = 1,...,N_{s}, \ j=1,\dots,m, \ k = 0,...,N_{cc}. \nonumber 
    \end{align}%     
\end{subequations}%
\allowdisplaybreaks[0]%
Greek letters listed on the left side of \eqref{primal feasibility start}-\eqref{primal feasibility end} are dual variables of the corresponding constraints. $\tau, \lambda_{j}$ and $s_{ij}$ are introduced into \eqref{forward DRO} during the reformulation of \eqref{general CC constraint}. It is important to note that the primal variable $x$ appears not only in \eqref{primal feasibility start} but also in \eqref{ab1}-\eqref{ab3}, as $b_k$ is a function of $x$. The convex DRO formulation in \eqref{forward DRO}, referred to as $\textbf{F-DRO}$, will be the FO problem in this paper. The unknown parameter $\epsilon$ in F-DRO appears only in \eqref{CVaR constraints}, which simplifies the task of parameter recovery through IO, as discussed in \cite{chan2020inverse}.  Specifically, \eqref{CVaR constraints} is referred to as the CVaR constraint because the left-hand side of \eqref{CVaR constraints} quantifies the conditional value-at-risk. The KKT conditions of \eqref{forward DRO} are:
\allowdisplaybreaks%
\begin{subequations} \label{KKT conditions}
    \begin{align}
        & c +A^{\top}\theta {\rm{+}} \sum_{i=1}^{N_s} \sum_{j=1}^{m} \sum_{k=0}^{N_{cc}} (\phi_{ikj}^{1} + \phi_{ikj}^{2}+\phi_{ikj}^{3})\frac{\partial b_{k}}{\partial x} = 0 \label{stationarity start}\\ 
        & \mu - \sum_{i=1}^{N_s} \sum_{j=1}^{m}\sum_{k=1}^{N_{cc}} \left(\phi_{ikj}^{1} + \phi_{ikj}^{2} + \phi_{ikj}^{3}\right) = 0\\
        & \frac{\mu \epsilon}{\gamma} {\rm{-}} \sum_{i=1}^{N_s} \sum_{k=0}^{N_{cc}} \left(\phi_{ikj}^{2} \vert \overline{\xi}_{j} {\rm{-}} \hat{\xi}_{ij}\vert {\rm{+}} \phi_{ikj}^{3} \vert \underline{\xi}_{j} {\rm{-}} \hat{\xi}_{ij}\vert \right) {\rm{-}}\eta_{j} =0 \\ 
        & \frac{\mu }{\gamma N_s} - \sum_{k=0}^{N_{cc}} \left(\phi_{ikj}^{1} + \phi_{ikj}^{2} + \phi_{ikj}^{3}\right) = 0 \label{stationarity end} \\
        & \theta^{\top} (Ax-h) = 0 \label{complementary slackness start} \\ 
        & \mu \left(\tau + \frac{1}{\gamma}\sum\nolimits_{j=1}^{m} \left( \lambda_{j} \epsilon + \frac{1}{N_s} \sum\nolimits_{i=1}^{N_s} s_{ij} \right) \right) = 0 \label{CVaR complementary slackness}\\
        & \phi_{ikj}^{1} \left( a_{kj} \hat{\xi}_{ij} + b_k - s_{ij}\right) = 0 \\
        & \phi_{ikj}^{2} \left( a_{kj} \overline{\xi}_{j} + b_k - \lambda_j \vert \overline{\xi}_{j} - \hat{\xi}_{ij}\vert -s_{ij} \right) = 0 \\
        & \phi_{ikj}^{3} \left( a_{kj} \underline{\xi}_{j} + b_k - \lambda_j \vert \underline{\xi}_{j} - \hat{\xi}_{ij}\vert -s_{ij} \right) = 0 \\
        & \eta_j \lambda_j = 0 \label{complementary slackness end} \\
        & \eqref{primal feasibility start}-\eqref{primal feasibility end} \label{primal feasibility} \\
        & \theta, \mu, \phi_{ikj}^{1}, \phi_{ikj}^{2}, \phi_{ikj}^{3}, \eta_j \geq 0 \label{dual feasibility} \\
        & i = 1,\dots,N_s,\ j = 1,\dots,m, \ k = 0,\dots,N_{cc}, \nonumber
    \end{align}%
\end{subequations}%

where \eqref{stationarity start}-\eqref{stationarity end} are the stationarity constraints, \eqref{complementary slackness start}-\eqref{complementary slackness end} are the complementary slackness, \eqref{primal feasibility} and \eqref{dual feasibility} stand for the primal feasibility and dual feasibility, respectively.

Although the F-DRO model, as detailed in \eqref{forward DRO}, originates from the linear chance-constrained model in \eqref{general LP with CC}, where the uncertainty $\xi$ does not influence the objective value, its convexity is preserved even when the objective function in \eqref{general objective} and the constraint in \eqref{general normal constraint} are convex but not necessarily linear, or when $\xi$ directly affects the objective function. In light of this, we offer the following remark: 
\begin{myremark} \label{other type of FO}
    F-DRO in \eqref{forward DRO} can be extended to address decision-making problems with a risk-averse objective function. Specifically, consider augmenting the objective function in \eqref{general objective} with a term for the worst-case expected cost, expressed as $\sup_{\mathbb{P}\in \mathcal{B}_{\epsilon}(\hat{\mathbb{P}}_{N_s})} \mathbb{E}^{\mathbb{P}} \left[f(x, \xi)\right]$, where $f$ is convex with respect to both $x$ and $\xi$. Following the approach in \cite{mohajerin2018data}, this adaptation can be formulated into a convex F-DRO problem with a similar structure as \eqref{forward DRO}.
\end{myremark}

\section{Inverse DRO Formulation}
\label{sec:Proposed_model}
We consider the following scenario: one party (the decision-maker) has made a decision using F-DRO in \eqref{forward DRO} (or other equivalent models). Another party (the data-miner) observes this decision and aims to infer the conservativeness level of the decision-maker, specifically, the value of $\epsilon$, by employing the inverse optimization method. This section is designed to assist the data-miner in achieving this parameter recovery task through the development of an inverse DRO (\textbf{I-DRO}) model. Fig.~\ref{fig:I-DRO} shows these interactions between the decision-maker and the data-miner, where $\epsilon^{\text{true}}$ is the conservativeness level chosen by the decision-maker in F-DRO, and $\epsilon^*$ is the value of $\epsilon$ recovered by the data-miner through I-DRO.
\begin{figure}[!htbp]
    \centering
    \includegraphics[width=0.9\linewidth]{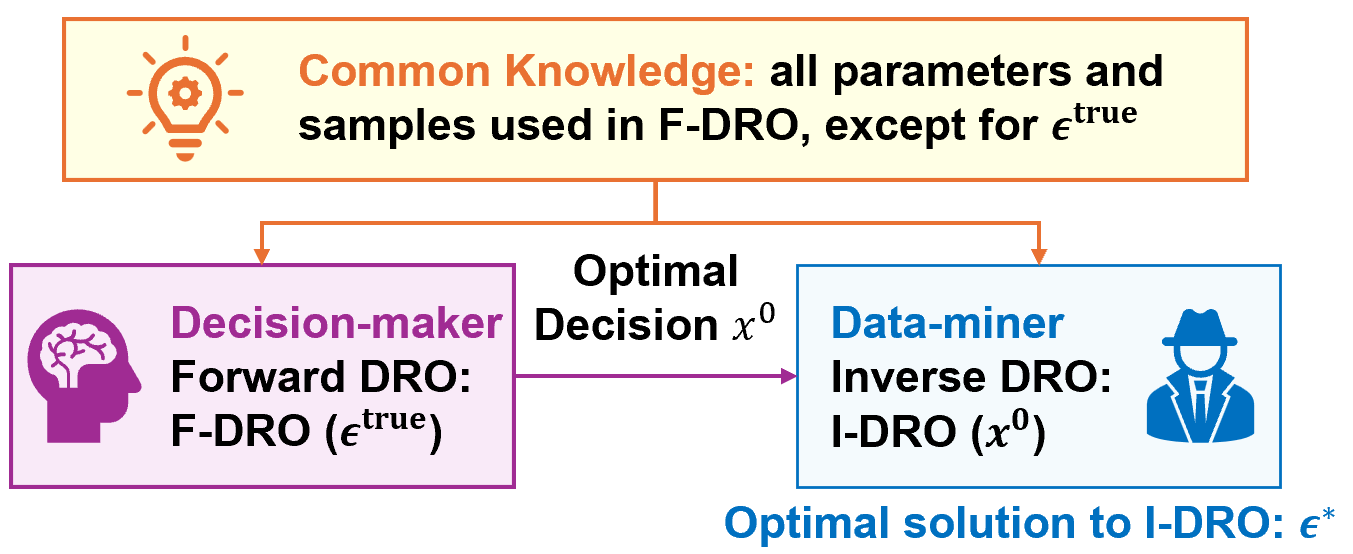}
    \caption{Forward and Inverse DRO Frameworks.}
    \label{fig:I-DRO}
\end{figure}

\subsection{Partially Constrained I-DRO Problem}
\label{subsec:IDRO_model}
Before formulating the I-DRO model, we need to clarify four assumptions: first, we assume the data-miner possesses full knowledge of or can accurately infer all constant parameters within F-DRO in \eqref{forward DRO}, such as $A$, $h$, $B$, $D$, $d$, $c$, and $\gamma$; second, we assume that both parties have access to the same historical dataset of the random variable $\xi$, denoted as $\{\hat{\xi_i}\}_{i=1}^{N_s}$; third, we assume that the data-miner can observe decision $x^0$ made by the decision-maker, and $x^0$ is the exact optimal solution of F-DRO in \eqref{forward DRO}; finally, we assume that the decision-maker selects the parameter in the F-DRO model appropriately so that the CVaR constraint \eqref{CVaR constraints} is binding.

These assumptions hold true for a wide range of practical scenarios. For example, consider the power market clearing process as the F-DRO problem. The parameters in F-DRO are either publicly available to market participants (e.g., power grid topology) or can be inferred from market-clearing results using data-driven approaches (e.g., production costs of generators), see \cite{ruiz2013revealing,chen2017strategic,birge2017inverse,liang2023data}. The historical data on power demand and renewable power generation are promptly and accurately published \cite{NYISO_load}. Additionally, \cite{liang2023data} indicates that the noise in market-clearing results is sparse and traceable, supporting the assumption that the observed market outputs are the optimal solutions of the forward market-clearing model.
Although there is no guarantee that the CVaR constraint will always be binding in practice, the likelihood of it being binding increases with properly tuned chance constraint parameters, as demonstrated in \cite{CCtuningHou2020, hou2022data}.

Despite these assumptions, I-DRO remains a partially constrained inverse problem. This arises from the fact that three groups of primal variables in F-DRO \eqref{forward DRO}, i.e., $\tau$, $\lambda_j$, and $s_{ij}$, are not directly measurable since they are auxiliary variables. Nevertheless, the risk-aware decision $x$ is normally observable. 
With the observed decision vector $x^0$, we can formulate the I-DRO model as:
\begin{subequations}\label{inverse DRO}
    \begin{align}
        \textbf{I-DRO}(x^0): & \quad \min_{\epsilon, p, d} \qquad  \vert \epsilon - \bar{\epsilon} \vert \label{IDRO objective}\\
        \mathrm{s.t.} \qquad & \quad \eqref{stationarity start}-\eqref{dual feasibility} \nonumber \\
        % & \quad  0 \leq \epsilon \leq \bar \epsilon \label{IDRO epsilon} \\
        & \quad p \in \hat{\mathcal{P}} \label{IDRO constraint p},
    \end{align}
\end{subequations}
where $p$ and $d$ respectively collect all the primal decisions, including $x,\tau,\{\lambda_j\}_{\forall j},\{s_{ij}\}_{\forall i,j}$, and dual variables, including $\theta, \mu, \{\phi_{ikj}^{1}, \phi_{ikj}^{2}, \phi_{ikj}^{3}\}_{\forall i,j,k}, \{\eta_j\}_{\forall j}$, in F-DRO \eqref{forward DRO}. Moreover, $\hat{\mathcal{P}}$ collects all the possible values of $p$ while enforcing $x=x^0$. 

Unlike $\bar \omega$ in \eqref{FO_kkt:obj} and \eqref{FO_kkt_partial:obj}, which is defined as an initial belief about vector $\omega$, $\bar{\epsilon}$ in \eqref{IDRO objective} is a predefined large positive number that serves as the upper bound for scalar $\epsilon$. This distinction arises from our assumption regarding the binding status of the CVaR constraint \eqref{CVaR constraints} in F-DRO. 
According to our assumption, $\epsilon^{\text{true}}$ renders \eqref{CVaR constraints} binding in F-DRO. In I-DRO, the complimentary slackness constraint \eqref{CVaR complementary slackness} can be satisfied either when \eqref{CVaR constraints} is binding ($\epsilon^* = \epsilon^{\text{true}}$, $\mu>0$) or non-binding ($\epsilon^* < \epsilon^{\text{true}}$, $\mu=0$), so the feasible region of $\epsilon^*$ is $[0, \epsilon^{\text{true}}]$. Thus, employing a large value of $\bar \epsilon$ in the I-DRO objective \eqref{IDRO objective} helps guide $\epsilon^*$ towards its optimal value within its feasible region. This goal could also be achieved by setting the I-DRO objective to $\max_{\epsilon} \epsilon$; however, this approach risks yielding unbounded results from I-DRO. A good choice of $\bar{\epsilon}$ will be discussed in Remark~\ref{epsilon value}.

\subsection{Properties of I-DRO Solutions}
\label{subsec:Properties I-DRO}
I-DRO in \eqref{inverse DRO} is essentially a bi-linear program, despite the presence of multi-linear terms involving the product of three variables ($\mu \epsilon \lambda_j, \forall j$) in \eqref{CVaR complementary slackness}. This is because each multi-linear term can be equivalently replaced by two bi-linear variables. Bi-linear programs can be solved using off-the-shelf solvers such as Gurobi \cite{gurobi}, but they typically do not guarantee the existence of a unique solution. In this subsection, we will explore the properties of the I-DRO solution $\epsilon^*$ and show that I-DRO in \eqref{inverse DRO} exhibits more favorable properties compared to general bi-linear programs.

We first propose and prove the following theorem for the existence and uniqueness of $\epsilon^{*}$:
\begin{myth} \label{existence of the solution} (Existence and Uniqueness of $\epsilon^{*}$) 
    For each observed decision $x^0$, I-DRO in \eqref{inverse DRO} guarantees the existence of a unique optimal solution $\epsilon^{*}$, assuming that $x^0$ is the exact optimal solution of F-DRO in \eqref{forward DRO} and the value of $\bar{\epsilon}$ is set sufficiently large. 
\end{myth}
\begin{proof}
    We first prove the existence of $\epsilon^{*}$. Given that F-DRO in \eqref{forward DRO} is convex, the KKT conditions in \eqref{KKT conditions} are necessary for optimality. Assuming the observed decision $x^0$ to be the exact optimal solution of F-DRO in \eqref{forward DRO} indicates that the KKT conditions in \eqref{KKT conditions} are satisfied for at least one set of primal and dual variables, denoted as $p^*$ and $d^*$, for a particular $\epsilon$. 
    Hence, there exists at least one $\epsilon$ that meets all the constraints in \eqref{inverse DRO}, even when we require $p=p^*$ and $d=d^*$, provided that the upper limit $\bar{\epsilon}$ is sufficiently large. 
    Specifically, \eqref{IDRO constraint p} prescribes values for only a subset of elements in vector $p$ and imposes no constraints on vector $d$, making it less restrictive than $p=p^*$ and $d=d^*$.

    Proving the uniqueness of $\epsilon^{*}$ is straightforward. Given that $\epsilon$ is a bounded scalar, the loss function \eqref{IDRO objective} is monotone, thereby rendering $\epsilon^{*}$ unique.
\end{proof}

However, the fact that $\epsilon^{*}$ is unique does not ensure that it precisely matches $\epsilon^{\text{true}}$ in F-DRO, i.e., the correctness of $\epsilon^{*}$ is not guaranteed.
Next, we will explore the conditions for $\epsilon^{*}$ to be correct. We first introduce a critical value $\epsilon^{\max}$:
\begin{mydef} 
Consider a random variable $\xi$ within domain $\Xi$, where $\overline{\xi}$ and $\underline{\xi}$ denote the upper and lower bounds of $\Xi$, respectively. We define:
\begin{equation} \label{definition of epsilon_max}
    \epsilon^{\max} = \max\left[ d_{W}(\hat{\mathbb{P}}_{N_s}, \delta_{\overline{\xi}}), d_{W}(\hat{\mathbb{P}}_{N_s}, \delta_{\underline{\xi}}) \right],
\end{equation}
where $\hat{\mathbb{P}}_{N_s}$ is the empirical distribution of $\xi$ constructed following \eqref{empirical distribution}, $\delta_{\overline{\xi}}$ and $\delta_{\underline{\xi}}$ are two Dirac distributions centered at $\overline{\xi}$ and $\underline{\xi}$, respectively. 
\end{mydef} 

For any $\mathbb{Q} \in \mathcal{M}(\Xi)$, it holds that $d_{W}(\hat{\mathbb{P}}_{N_s}, \mathbb{Q}) \leq \epsilon^{\max}$. This suggests that every possible distribution within $\mathcal{M}(\Xi)$ falls into the Wasserstein ball centered at $\hat{\mathbb{P}}_{N_s}$ with a radius of $\epsilon^{\max}$, i.e., $\mathcal{M}(\Xi) \subseteq \mathcal{B}_{\epsilon^{\max}}(\hat{\mathbb{P}}_{N_s})$. 

In the following theorem, we establish the necessary and sufficient conditions for $\epsilon^{*}$ to precisely match $\epsilon^{\text{true}}$: 
\begin{myth} \label{correctness of the solution}
    (Correctness of $\epsilon^{*}$) Assuming $\bar \epsilon$ in \eqref{IDRO objective} is chosen to exceed $\epsilon^{\max}$, the value of $\epsilon^{*}$ derived from I-DRO in \eqref{inverse DRO} matches the value of $\epsilon^{\text{true}}$ used in F-DRO in \eqref{forward DRO} if and only if both of the following conditions are met: (i) $\epsilon^{\text{true}} <\epsilon^{\max}$, and (ii) the CVaR constraint (11c) is binding during the forward decision-making process, i.e., $\tau^{*} + \frac{1}{\gamma}\sum_{j=1}^{m} (\lambda_{j}^{*} \epsilon^{\text{true}} + \frac{1}{N_s} \sum_{i=1}^{N_s} s_{ij}^{*}) = 0$, where $\tau^{*}, \lambda_{j}^{*}, s_{ij}^{*}$ are the optimal solutions to F-DRO in (11).
\end{myth} 
\begin{proof}
We first prove that $\epsilon^{\text{true}} < \epsilon^{\max}$ is a necessary condition for $\epsilon^{*} = \epsilon^{\text{true}}$. It is equivalent to show that $\epsilon^{\text{true}} \geq \epsilon^{\max}$ leads to $\epsilon^{*} \neq \epsilon^{\text{true}}$. Given that $\mathcal{M}(\Xi) \subseteq \mathcal{B}_{\epsilon^{\max}}(\hat{\mathbb{P}}_{N_s})$, for any $\epsilon^{\text{true}} \geq \epsilon^{\max}$, the optimal solution to F-DRO in  \eqref{forward DRO} remains unaffected by variations in $\epsilon^{\text{true}}$, leading to an unchanged input $x_0$ for I-DRO in \eqref{inverse DRO}. Consequently, the value of $\epsilon^{*}$ stays consistent for any $\epsilon^{\text{true}} \geq \epsilon^{\max}$, indicating that $\epsilon^{*} \neq \epsilon^{\text{true}}$.

Similarly, we can prove that the binding state of \eqref{CVaR constraints} is necessary for $\epsilon^{*} = \epsilon^{\text{true}}$. If \eqref{CVaR constraints} is not binding at $\epsilon^{\text{true}}$, there exists an $\epsilon^+ > \epsilon^{\text{true}}$ that also satisfies \eqref{CVaR constraints}. Therefore, both $\epsilon^+$ and $\epsilon^{\text{true}}$ will yield an identical optimal solution to F-DRO in \eqref{forward DRO}, and they are both feasible solutions to I-DRO in \eqref{inverse DRO}. Given that $\bar {\epsilon} > \epsilon^+ > \epsilon^{\text{true}}$, 
I-DRO in \eqref{inverse DRO} will choose $\epsilon^* = \epsilon^+$ to minimize the objective, which means $\epsilon^{*} \neq \epsilon^{\text{true}}$.

Then, we prove that $\epsilon^{\text{true}} < \epsilon^{\max}$ combined with \eqref{CVaR constraints} being binding is sufficient for $\epsilon^{*} = \epsilon^{\text{true}}$. 
When $\epsilon^{\text{true}} < \epsilon^{\max}$, $\mathcal{B}_{\epsilon^{\text{true}}}(\hat{\mathbb{P}}_{N_s})$ is a subset of $\mathcal{M}(\Xi)$ and changes with variations in $\epsilon^{\text{true}}$.
Moreover, with \eqref{CVaR constraints} binding at $\epsilon^{\text{true}}$, it indicates that any $\epsilon^- \leq \epsilon^{\text{true}}$ would render \eqref{CVaR constraints} feasible, making it a feasible solution to I-DRO in  \eqref{inverse DRO}. Since $\bar {\epsilon} > \epsilon^{\text{true}} \geq \epsilon^{-}$, the optimal solution to \eqref{inverse DRO} will be $\epsilon^* = \epsilon^{\text{true}}$.\end{proof}

Condition (i) in Theorem~\ref{correctness of the solution} can also be interpreted through a decision making lens. A decision-maker with $\epsilon^{\text{true}} \geq \epsilon^{\max}$ exhibits a highly conservative risk attitude by relying solely on the bounds of $\Xi$, transforming F-DRO in \eqref{forward DRO} into a robust optimization, possibly due to a belief that historical samples $\{\hat{\xi}_{i}\}_{i=1}^{N_s}$ do not provide sufficient insights beyond the bounds $\overline{\xi}$ and $\underline{\xi}$. This suggests that less representative samples can result in overly conservative decisions, thereby impacting the successful recovery of $\epsilon$.

Moreover, condition (ii) in Theorem~\ref{correctness of the solution} aligns with findings from other studies that investigate the recovery of constraint parameters via inverse optimization, such as \cite{chan2020inverse}. A constraint must be binding to influence the optimization outcomes, so the unknown parameter within a constraint can only be precisely recovered when the observed decision occurs while this constraint is binding.

The two conditions in Theorem~\ref{correctness of the solution} reflect the decision-maker's perspective, but the data-miner who employs I-DRO may lack the prior knowledge needed to verify these conditions. Therefore, it is crucial to allow users of I-DRO to assess the correctness of I-DRO results without knowing if both conditions have been met.
If condition (i) is not met, i.e., if $\epsilon^{\text{true}}> \epsilon^{\max}$, then any $\epsilon > \epsilon^{\max}$ becomes a feasible solution to I-DRO, leading to $\epsilon^* = \bar \epsilon$. Let $\hat \epsilon$ denote the value of $\epsilon$ that renders \eqref{CVaR constraints} binding in F-DRO. If Condition (ii) is not met, i.e., $\epsilon^{\text{true}}< \hat \epsilon$, the potential I-DRO result needs to be discussed in two cases. If $\hat \epsilon \le \bar \epsilon$, then $\hat \epsilon$ is the largest $\epsilon$ satisfying I-DRO constraints, leading to  $\epsilon^* = \hat \epsilon$. If $\hat \epsilon > \bar \epsilon$, then $\bar \epsilon$ falls within the feasible region of I-DRO, leading to  $\epsilon^* = \bar \epsilon$. In summary, the result of $\epsilon^* = \bar \epsilon$ typically suggests that one of the two conditions in Theorem~\ref{correctness of the solution} is not met, indicating a failure of the I-DRO process. Therefore, we suggest the following rule for choosing $\bar{\epsilon}$:
\begin{myremark} \label{epsilon value}
The upper limit $\bar{\epsilon}$ in \eqref{IDRO objective} should be set such that $\bar{\epsilon} \geq \epsilon^{\max}$. Given that $\epsilon^{\max}$ can be calculated from historical samples of $\xi$ and the data-miner implementing I-DRO in \eqref{inverse DRO} has access to these samples, setting $\bar{\epsilon}$ in this manner is practical. It is also advisable to assign a specific value to $\bar \epsilon$ as $\epsilon^* = \bar \epsilon$ is an indicator of potential I-DRO failure.
\end{myremark} 

\subsection{Extensions of I-DRO Model}
\label{subsec:IDRO_extension}
We do not want the strong assumptions made at the beginning of Section~\ref{subsec:IDRO_model} to restrict the application scenarios of I-DRO. In fact, with some straightforward modifications to the I-DRO framework, it is possible to relax these assumptions. 

Regarding the first assumption, we have the following remark discussing I-DRO with extended objective:
\begin{myremark} \label{data-driven}
   When parameters other than $\epsilon$ in F-DRO are unknown, the I-DRO model presented in \eqref{inverse DRO} can be adapted by treating these unknown parameters as decision variables instead of constants. For instance, if $c$ is unknown, the revised objective function for I-DRO would be $\min_{\epsilon, c, p, d} \ \vert \epsilon - \bar \epsilon \vert + w \Vert c - \bar c \Vert$. Here, $\bar c$ represents a predefined parameter that reflects the initial estimate of $c$, and $w$ is a weighting factor that balances the two parts of the objective function.
\end{myremark}

Regarding the third assumption, we have the following remark discussing I-DRO in noisy settings:
\begin{myremark} \label{noisy situation}
   We can modify I-DRO in \eqref{inverse DRO} if $x^0$ is feasible but not necessarily optimal for F-DRO in \eqref{forward DRO}. The modifications include: i) introducing slack variables to relax the KKT conditions in \eqref{stationarity start} - \eqref{complementary slackness end}, and ii) incorporating the absolute values of these slack variables into the objective function in \eqref{IDRO objective} to account for the deviations from the optimum \cite{aswani2018inverse}. 
\end{myremark}

Instead of observing a single decision, we typically monitor multiple decisions made by the same decision-maker over time and under various circumstances. This enables a data-driven I-DRO approach, similar to the data-driven IO model proposed in \cite{mohajerin2018dataio}. In this case, the fourth assumption, also the second necessary and sufficient condition in Theorem~\ref{correctness of the solution}, can be relaxed:
\begin{myremark} \label{non-binding}
   In the data-driven I-DRO approach, it is sufficient for the CVaR constraint \eqref{CVaR constraints} to be binding at some point during a given period. Let $\epsilon^*_t$ represent the conservativeness level recovered based on observations at time $t$. If $\epsilon^*_t$ varies with changes in $t$, while $\epsilon^{\text{true}}$ remains constant, this indicates that at least one of the conditions in Theorem~\ref{correctness of the solution} is not constantly met. In such scenarios, the smallest value among all the I-DRO results $\epsilon^*_t$ should be chosen as the final recovered conservativeness level.
\end{myremark}

We also want to highlight that the I-DRO model is adaptable to scenarios where the forward decision-making model is not specifically a DRO model, but shares similar properties with DRO, such as aiming to mitigate the risk of worst-case losses. In this case, the F-DRO model in \eqref{forward DRO} can serve as a surrogate for the decision-making process, and the recovered value $\epsilon^*$ is no longer the Wasserstein metric. Nevertheless, the relative value of $\epsilon^*$ reflects the conservativeness level of a decision-maker. The larger the $\epsilon$, the more conservative the decision.

In summary, three extreme scenarios could lead to incorrect estimations of $\epsilon$ through I-DRO in \eqref{inverse DRO}. These scenarios include: i) using an unsuitable $\epsilon$ in F-DRO, ii) the chance constraints in F-DRO not being binding, and iii) the historical samples used in F-DRO lacking representativeness, e.g., when the sample size is too small. In the next section, we will demonstrate I-DRO performance in both normal and extreme scenarios through numerical experiments.

\section{Numerical Experiments}
\label{sec:Case_study}
In this section, we formulate a chance-constrained DC optimal power flow (DC-OPF) model, routinely used by power system operators to schedule their operations, and identify it as F-DRO. We demonstrate how the conservativeness of system operators can be recovered through I-DRO. We perform computations on an illustrative IEEE 5-bus system using synthetic data to assess I-DRO robustness across the extreme scenarios described above and a realistic New York Independent System Operator (NYISO) 11-zone system to show its computational efficiency and practical value. All the numerical experiments are conducted using the Gurobi solver \cite{gurobi} in Python and are performed on a PC with an Intel Core i9, 2.20 GHz with 16 GB RAM.

\subsection{Chance-constrained DC-OPF model}
In power system scheduling, uncertainty primarily stems from forecast errors in net load (i.e., demand minus renewable energy injections), potentially resulting in constraint violations. To mitigate these violations, the chance-constrained DC-OPF model has been proposed \cite{bienstock2014chance}.

We denote $N_g$, $N_l$, and $N_b$ as the number of generators, transmission lines, and buses in the system, respectively. Vectors $c, x, r \in \mathbb{R}^{N_g}$ are the generation cost, output power, and reserved capacity of generators, respectively, $e \in \mathbb{R}^{N_b}$ collects the net load at each bus, $R$ is the overall reserve requirement set by system operators, $\Phi$ is the power transfer distribution factor (PTDF) matrix, and $S$ is the indicator matrix that relates generators to buses, $\bm{1}_{n}$ is an all-one vector of length $n$. Accordingly, the chance-constrained DC-OPF model is formulated as: 
\begin{subequations} \label{normal DCOPF}
    \begin{align} 
        \min_{x, r} \ &  c^{\top} x \label{normal DCOPF cost function}\\
        \mathrm{s.t.} \ & \bm{1}_{N_g}^{\top} x = \bm{1}_{N_b}^{\top} e \label{normal DCOPF power balance}\\
         &  x^{\min} \leq x+r \leq x^{\max}\label{normal DCOPF gen lb}\\
         % &  x+r \leq x^{\max} \label{normal DCOPF gen ub}\\
         & \bm{1}_{N_g}^{\top} r \geq R \label{normal DCOPF reserve}\\
         & \inf_{\mathbb{P} \in \mathcal{B}_{\epsilon}} \text{Pr}\left\{
            \begin{array}{c}
                 {\rm{-}}\Phi (Sx{\rm{-}}e) \leq f^{\max}{\rm{+}}\Delta f\\
                  \Phi (Sx{\rm{-}}e) \leq f^{\max} {\rm{+}} \Delta f 
            \end{array}
         \right\} \geq 1{\rm{-}}\gamma \label{cc of power flow} \\
         &  x, r \geq 0 , \label{normal DCOPF positive}
    \end{align}
\end{subequations}
where \eqref{normal DCOPF cost function} minimizes the total generation cost, \eqref{normal DCOPF power balance} enforces the power balance, \eqref{normal DCOPF gen lb} ensures that the power output and scheduled reserve of generators are within the bounds $x^{\min}$ and $x^{\max}$, \eqref{normal DCOPF reserve} enforces the total reserve requirement in the system. The chance constraint in \eqref{cc of power flow} guarantees the worst-case violation rate of the power flow constraint is less than $\gamma$, where $f^{\max}$ is the vector of maximum power flow limit, and $\Delta f$ is a random vector. Eq.~\eqref{normal DCOPF} may differ from chance-constrained DC-OPF models that account for the net load uncertainty $\Delta e$ as a random vector \cite{mieth2020risk}. However, these formulations are equivalent since $\Delta e$ can be converted into the uncertainty in power flow as $\Delta f = -\Phi \Delta e$. The the possible distributions of $\Delta f$ satisfy $\mathbb{P} \in \mathcal{B}_{\epsilon} (\hat{\mathbb{P}}_{N_s})$, where $\{\Delta f_{i}\}_{i=1}^{N_s}$ are the historical samples of $\Delta f$. We set $\gamma = 0.05$ and $R = 0.05 e$ in the case study. 

\subsection{Illustrative Example: IEEE 5-bus System}
We first apply the I-DRO method to the IEEE 5-bus system adapted from Pandapower \cite{pandapower2018}.
We generate $N_s =100$ random samples of $\Delta f$ based on a uniform distribution around $f^{\max}$ with a radius of $0.1 f^{\max}$. Using these samples, we determine $\epsilon^{\max}= 0.4036$ as per \eqref{definition of epsilon_max}. 
In line with Remark~\ref{epsilon value}, we set $\bar{\epsilon}= 100$, i.e., much greater than $\epsilon^{\max}$.

Next, we will consider three extreme scenarios that could potentially challenge the effectiveness of I-DRO:
\subsubsection{Scenario 1: Overly Conservative Decision-Making}
We first evaluate the correctness of recovered parameter $\epsilon^{*}$ through I-DRO, against varying levels of decision conservativeness $\epsilon^{\text{true}}$ in F-DRO. The results in Table~\ref{recovered results with different epi} indicate that when $\epsilon^{\text{true}} < \epsilon^{\max}$, the recovered $\epsilon^{*}$ matches $ \epsilon^{\text{true}}$, highlighting I-DRO effectiveness. Conversely, if $\epsilon^{\text{true}} \geq \epsilon^{\max}$, the first necessary condition in Theorem~\ref{correctness of the solution} is violated, so the recovered $\epsilon^{*} = \bar{\epsilon}$ and does not match $\epsilon^{\text{true}}$.

\subsubsection{Scenario 2: Chance Constraints Not Binding}
Next, we explore the scenario when the second condition in Theorem~\ref{correctness of the solution} does not hold. We increase the line capacity limit $f^{\max}$ to different levels so that violating the chance constraint \eqref{cc of power flow} in F-DRO is nearly impossible, given the distribution of $\Delta f$ is within $\mathcal{B}_{\epsilon^{\text{true}}} (\hat{\mathbb{P}}_{N_s})$. We keep $\epsilon^{\text{true}}=0.01$ and $N_s=100$ in this case. The results in Table~\ref{results with different f_max} show that I-DRO fails when $f^{\max}$ is tripled relative to its original value.

\subsubsection{Scenario 3: Historical Samples Not Representative}
We further examine the case when the samples are scarce and inadequate to represent the distribution of $\Delta f$. In this scenario, we set $\epsilon^{\text{true}}= 0.01$ and maintain $f^{\max}$ at its original value. The results in Table~\ref{recovered results with different sample sizes} demonstrate that the I-DRO approach fails when the sample size is 15 or fewer.

\begin{table}[!t]
    \centering
    \renewcommand{\arraystretch}{1.2}
    \caption{I-DRO results under different $\epsilon^{\text{true}}$}
    \begin{tabular}[width=0.8\textwidth]{c|ccccccc}
       \hline
       $\epsilon^{\text{true}}$  & 0 & 0.01 & 0.05 & 0.1 & 0.2 & $\epsilon^{\max}$  & 0.5 \\ 
       \hline
       $\epsilon^{*}$ & 0 & 0.01 & 0.05 & 0.1 & 0.2 & 100 & 100 \\
       \hline
    \end{tabular}
    \label{recovered results with different epi}
\end{table}

\begin{table}[!t]
    \centering
    \renewcommand{\arraystretch}{1.2}
    \caption{I-DRO results under different $f^{\max}$}
    \begin{tabular}{c|ccccc}
        \hline
         $f^{\max}$ & $\times 0.9$ & $\times 1$ & $\times 1.1$ & $\times 3$ & $\times 5$\\
         \hline
         $\epsilon^{*}$ & 0.01 & 0.01 & 0.01 & 100 & 100 \\
        \hline
    \end{tabular}
    \label{results with different f_max}
\end{table}

\begin{table}[!t]
    \centering
    \renewcommand{\arraystretch}{1.2}
    \caption{I-DRO results under different sample size $N_s$} 
    \begin{tabular}{c|cccccc}
        \hline
         $N_s$&  100 & 75 & 50 & 25 & 15 & 10\\
         \hline
         $\epsilon^{*}$ & 0.01 & 0.01 & 0.01 & 0.01 & 100 & 100\\
         \hline
    \end{tabular}
    \label{recovered results with different sample sizes}
\end{table}

\subsection{Tests in NYISO 11-Zone System}
While we identify three scenarios where I-DRO could fail, they are rare in practice. Decision-makers usually gather ample historical data before using F-DRO, and sufficient data availability allows for a reasonable level of conservativeness. Thus, the chance of facing the first and third extreme scenarios in real-world problems is low. However, the second extreme scenario may arise in practical settings, and we examine it in a realistic NYISO system with 11 nodes, 13 lines, and 335 generators. The parameters of this system are estimated from publicly available databases \cite{NYISO_Gold_Book}, and the samples of $\Delta f$ are derived based on actual data  \cite{NYISO_load}. 

We applied I-DRO to the data on August 2023 and January 2023, which represent the peak and off-peak load conditions within a year, respectively. For January, we set $\epsilon^{\text{true}}=1$. For August, to accommodate the increased likelihood of line capacity limits being violated, we adjusted it to a less conservative value of $\epsilon^{\text{true}}=0.5$ to ensure system feasibility. I-DRO is applied hourly for each day, employing 100 historical samples that match the specific hour in the respective months. This process results in a total of 24 runs for each day. The I-DRO results indicate that the recovered $\epsilon^{*}$ matches $\epsilon^{\text{true}}$ in all instances, confirming the absence of the extreme scenarios in the real system. Moreover, we summarize the computation times of I-DRO across the 48 runs in Table \ref{results in NYISO}. The average computation time is under 400 seconds, demonstrating the efficiency and scalability of our approach.
\begin{table}[!t]
    \centering
    \renewcommand{\arraystretch}{1.2}
    \caption{Computing times of I-DRO in NYISO system }
    \label{results in NYISO}
    \begin{tabular}{cccc}
    \hline
         Case & Average time (s) & Maximum time (s) & Variations (s)\\
        \hline
        January & 398.6 & 2863.2 & 188.3 \\
        August & 400.8 & 1437.0 & 297.0 \\
         \hline
    \end{tabular} 
\end{table}

\section{Conclusion and Future Work}
\label{sec:conclusion}
This paper presents an I-DRO approach for recovering the conservativeness level of decision-makers from optimal decisions made through F-DRO. We cast the I-DRO problem as a bi-linear program, ensuring it can be solved using conventional solvers, while also providing theoretical guarantees regarding its performance.
Numerical experiments on the IEEE 5-bus system not only show the effectiveness of I-DRO under normal parameter settings but also highlight three extreme scenarios that might lead to I-DRO failures. Nevertheless, tests on the NYISO system show these extreme scenarios to be infrequent in real-world problems. The expedited computation times further highlight the I-DRO efficiency and adaptability for broader applications. 

A limitation of I-DRO is its reliance on the assumption of complete knowledge about all parameters and samples used in F-DRO, as well as assumptions concerning the accuracy of observations and the binding status of the CVaR constraint. While these assumptions are reasonable for power system applications, they may not be valid in other real-world scenarios. Future studies will focus on relaxing these assumptions to explore the performance of data-driven I-DRO in environments with limited prior information and increased noise.

\bibliographystyle{IEEEtran}
\bibliography{ref}

\begin{thebibliography}{10}
\providecommand{\url}[1]{#1}
\csname url@samestyle\endcsname
\providecommand{\newblock}{\relax}
\providecommand{\bibinfo}[2]{#2}
\providecommand{\BIBentrySTDinterwordspacing}{\spaceskip=0pt\relax}
\providecommand{\BIBentryALTinterwordstretchfactor}{4}
\providecommand{\BIBentryALTinterwordspacing}{\spaceskip=\fontdimen2\font plus
\BIBentryALTinterwordstretchfactor\fontdimen3\font minus \fontdimen4\font\relax}
\providecommand{\BIBforeignlanguage}[2]{{%
\expandafter\ifx\csname l@#1\endcsname\relax
\typeout{** WARNING: IEEEtran.bst: No hyphenation pattern has been}%
\typeout{** loaded for the language `#1'. Using the pattern for}%
\typeout{** the default language instead.}%
\else
\language=\csname l@#1\endcsname
\fi
#2}}
\providecommand{\BIBdecl}{\relax}
\BIBdecl

\bibitem{mohajerin2018data}
P.~Mohajerin~Esfahani and D.~Kuhn, ``Data-driven distributionally robust optimization using the wasserstein metric: Performance guarantees and tractable reformulations,'' \emph{Math. Program.}, vol. 171, no.~1, pp. 115--166, 2018.

\bibitem{chan2023inverse}
T.~C. Chan, R.~Mahmood, and I.~Y. Zhu, ``Inverse optimization: Theory and applications,'' \emph{Oper. Res.}, 2023.

\bibitem{ruiz2013revealing}
C.~Ruiz, A.~J. Conejo, and D.~J. Bertsimas, ``Revealing rival marginal offer prices via inverse optimization,'' \emph{IEEE Transactions on Power Systems}, vol.~28, no.~3, pp. 3056--3064, 2013.

\bibitem{chen2017strategic}
R.~Chen, I.~C. Paschalidis, and M.~C. Caramanis, ``Strategic equilibrium bidding for electricity suppliers in a day-ahead market using inverse optimization,'' in \emph{56th IEEE Conference on Decision and Control (CDC 2017)}.\hskip 1em plus 0.5em minus 0.4em\relax Melbourne, Australia: IEEE, 2017, pp. 220--225.

\bibitem{liang2023data}
Z.~Liang and Y.~Dvorkin, ``Data-driven inverse optimization for marginal offer price recovery in electricity markets,'' in \emph{Proc. of the 14th ACM Int. Conf. on Future Energy Systems}, 2023, pp. 497--509.

\bibitem{birge2017inverse}
J.~R. Birge, A.~Horta{\c{c}}su, and J.~M. Pavlin, ``Inverse optimization for the recovery of market structure from market outcomes: An application to the miso electricity market,'' \emph{Operations Research}, vol.~65, no.~4, pp. 837--855, 2017.

\bibitem{saez2016data}
J.~Saez-Gallego, J.~M. Morales, M.~Zugno, and H.~Madsen, ``A data-driven bidding model for a cluster of price-responsive consumers of electricity,'' \emph{IEEE Transactions on Power Systems}, vol.~31, no.~6, pp. 5001--5011, 2016.

\bibitem{kovacs2021inverse}
A.~Kov{\'a}cs, ``Inverse optimization approach to the identification of electricity consumer models,'' \emph{Central European Journal of Operations Research}, vol.~29, no.~2, pp. 521--537, 2021.

\bibitem{ahuja2001inverse}
R.~K. Ahuja and J.~B. Orlin, ``Inverse optimization,'' \emph{Oper. Res.}, vol.~49, no.~5, pp. 771--783, 2001.

\bibitem{chan2020inverse}
T.~C. Chan and N.~Kaw, ``Inverse optimization for the recovery of constraint parameters,'' \emph{Eur. J. Oper. Res.}, vol. 282, no.~2, pp. 415--427, 2020.

\bibitem{aswani2018inverse}
A.~Aswani, Z.-J. Shen, and A.~Siddiq, ``Inverse optimization with noisy data,'' \emph{Oper. Res.}, vol.~66, no.~3, pp. 870--892, 2018.

\bibitem{mohajerin2018dataio}
P.~Mohajerin~Esfahani, S.~Shafieezadeh-Abadeh, G.~A. Hanasusanto, and D.~Kuhn, ``Data-driven inverse optimization with imperfect information,'' \emph{Mathematical Programming}, vol. 167, no.~1, pp. 191--234, 2018.

\bibitem{schaefer2009inverse}
A.~J. Schaefer, ``Inverse integer programming,'' \emph{Optimization Letters}, vol.~3, pp. 483--489, 2009.

\bibitem{hobbs2000strategic}
B.~F. Hobbs, C.~B. Metzler, and J.-S. Pang, ``Strategic gaming analysis for electric power systems: An mpec approach,'' \emph{IEEE transactions on power systems}, vol.~15, no.~2, pp. 638--645, 2000.

\bibitem{abbaspourtorbati2015swiss}
F.~Abbaspourtorbati and M.~Zima, ``The swiss reserve market: Stochastic programming in practice,'' \emph{IEEE Transactions on Power Systems}, vol.~31, no.~2, pp. 1188--1194, 2015.

\bibitem{mieth2020risk}
R.~Mieth, M.~Roveto, and Y.~Dvorkin, ``Risk trading in a chance-constrained stochastic electricity market,'' \emph{IEEE Control Syst. Lett.}, vol.~5, no.~1, pp. 199--204, 2020.

\bibitem{ghobadi2018robust}
K.~Ghobadi, T.~Lee, H.~Mahmoudzadeh, and D.~Terekhov, ``Robust inverse optimization,'' \emph{Oper. Res. Lett.}, vol.~46, no.~3, pp. 339--344, 2018.

\bibitem{dong2021wasserstein}
C.~Dong and B.~Zeng, ``Wasserstein distributionally robust inverse multiobjective optimization,'' in \emph{Proc. of the AAAI Conf. on Artificial Intelligence}, vol. 35-7, 2021, pp. 5914--5921.

\bibitem{chassein2018variable}
A.~Chassein and M.~Goerigk, ``Variable-sized uncertainty and inverse problems in robust optimization,'' \emph{Eur. J. Oper. Res.}, vol. 264, no.~1, pp. 17--28, 2018.

\bibitem{rockafellar2000optimization}
R.~T. Rockafellar and S.~Uryasev, ``Optimization of conditional value-at-risk,'' \emph{Journal of Risk}, vol.~2, pp. 21--42, 2000.

\bibitem{mieth2023data}
R.~Mieth, J.~M. Morales, and H.~V. Poor, ``Data valuation from data-driven optimization,'' \emph{arXiv preprint arXiv:2305.01775}, 2023.

\bibitem{NYISO_load}
{NYISO}, ``{NYISO load data},'' \url{https://www.nyiso.com/load-data}, 2024.

\bibitem{CCtuningHou2020}
A.~M. Hou and L.~A. Roald, ``Chance constraint tuning for optimal power flow,'' in \emph{2020 International Conference on Probabilistic Methods Applied to Power Systems (PMAPS)}, 2020, pp. 1--6.

\bibitem{hou2022data}
------, ``Data-driven tuning for chance constrained optimization: analysis and extensions,'' \emph{Top}, vol.~30, no.~3, pp. 649--682, 2022.

\bibitem{gurobi}
\BIBentryALTinterwordspacing
{Gurobi Optimization, LLC}, ``{Gurobi Optimizer Reference Manual},'' 2023. [Online]. Available: \url{https://www.gurobi.com}
\BIBentrySTDinterwordspacing

\bibitem{bienstock2014chance}
D.~Bienstock, M.~Chertkov, and S.~Harnett, ``Chance-constrained optimal power flow: Risk-aware network control under uncertainty,'' \emph{Siam Review}, vol.~56, no.~3, pp. 461--495, 2014.

\bibitem{pandapower2018}
L.~Thurner \emph{et~al.}, ``pandapower — an open-source python tool for convenient modeling, analysis, and optimization of electric power systems,'' \emph{IEEE Trans. Power Syst.}, vol.~33, no.~6, pp. 6510--6521, Nov 2018.

\bibitem{NYISO_Gold_Book}
{NYISO}, ``{Load and Capacity Data Gold Book},'' \url{https://www.nyiso.com/documents/20142/2226333/2023-Gold-Book-Public.pdf}, 2023.

\end{thebibliography}

\end{document}